\documentclass[a4paper,10pt]{amsproc}
\pdfoutput=1
\usepackage{amsfonts, amsmath, amssymb, graphicx, mathrsfs}

\theoremstyle{plain}

\newtheorem{theorem}{Theorem}[section]
\newtheorem{lemma}[theorem]{Lemma}
\newtheorem{proposition}[theorem]{Proposition}
\theoremstyle{definition}
\newtheorem{definition}[theorem]{Definition}

\newtheorem{counter example}[theorem]{Counter Example}

\newtheorem{corollary}[theorem]{Corollary}
\newtheorem{example}[theorem]{Example}
\newtheorem{question}[theorem]{Question}
\numberwithin{equation}{section}

\DeclareMathOperator{\Max}{Max}

\begin{document}

\title[Ideals in Rings and Intermediate Rings of Measurable Functions]{Ideals in Rings and Intermediate Rings of Measurable Functions}

\author[S.K.Acharyya]{Sudip Kumar Acharyya}
\address{Department of Pure Mathematics, University of Calcutta, 35, Ballygunge Circular Road, Kolkata 700019, West Bengal, India}
\email{sdpacharyya@gmail.com}

\author[S. Bag]{Sagarmoy Bag}
\address{Department of Pure Mathematics, University of Calcutta, 35, Ballygunge Circular Road, Kolkata 700019, West Bengal, India}
\email{sagarmoy.bag01@gmail.com}
\thanks{The second author thanks the NBHM, Mumbai-400 001, India, for financial support}

\author[J. Sack]{Joshua Sack} 
\address{Department of Mathematics and Statistics, California State University Long Beach, 1250 Bellflower Blvd, Long Beach, CA 90840, USA}
\email{joshua.sack@csulb.edu}

\thanks {}

%                                                                                                                           %
%         Please use the current 2010 Mathematics Subject Classification:             %
%         http://www.ams.org/mathscinet/msc/                                                        %
%         http://www.zentralblatt-math.org/msc/en/                                                 %
%%%%%%%%%%%%%%%%%%%%%%%%%%%%%%%%%%%%%%%%%%%%%%%%%%%

\thanks {}

\maketitle
\begin{abstract}
The set of all maximal ideals of the ring $\mathcal{M}(X,\mathcal{A})$ of real valued measurable functions on a measurable space $(X,\mathcal{A})$ equipped with the hull-kernel topology is shown to be homeomorphic to the set $\hat{X}$ of all ultrafilters of measurable sets on $X$ with the Stone-topology. This yields a complete description of the maximal ideals of $\mathcal{M}(X,\mathcal{A})$ in terms of the points of $\hat{X}$. It is further shown that the structure spaces of all the intermediate subrings of $\mathcal{M}(X,\mathcal{A})$ containing the bounded measurable functions are one and the same and are compact Hausdorff zero-dimensional spaces. 
It is observed that when $X$ is a $P$-space, then $C(X) = \mathcal{M}(X,\mathcal{A})$ where $\mathcal{A}$ is the $\sigma$-algebra consisting of the zero-sets of $X$.
\end{abstract}

\section{introduction}

In what follows $(X,\mathcal{A})$ stands for a nonempty set $X$ equipped with a family $\mathcal{A}$ of subsets of $X$, which is closed under countable union and complementation. Such a family $\mathcal{A}$ is known as a $\sigma$-algebra over $X$ and the pair $(X,\mathcal{A})$ is called a measurable space, and members of $\mathcal{A}$ are called $\mathcal{A}$-measurable sets.  
A function $f: X\mapsto \mathbb{R}$ is called $\mathcal{A}$-measurable if for any real number $\alpha$, $f^{-1}(\alpha,\infty)$ is a member of $\mathcal{A}$.  
It is a standard result in measure theory that the aggregate $\mathcal{M}(X,\mathcal{A})$ of all real valued $\mathcal{A}$-measurable functions on $X$, constitutes a commutative lattice ordered ring with unity if the relevant operations are defined point wise on $X$ \cite{AHM2009}. The chief object of study in this article is this ring $\mathcal{M}(X,\mathcal{A})$ together with some of its chosen subrings \emph{viz} those rings which contain all the bounded $\mathcal{A}$-measurable functions on $X$. The first paper concerning this ring dates back to 1966 \cite{G1966}.
It was followed by a series of articles in 1974, 1977, 1978, 1981 in \cite{V1974}, \cite{V1977}, \cite{V1978}, \cite{V1981}.
After a long gap of more than twenty five years, the articles \cite{AAMH2013}, \cite{AHM2009}, and \cite{EMY2018} appeared, which deal with various problems related to these rings. 

In Section~\ref{sec:idealsfilters}, we initiate a kind of duality between ideals (maximal ideals) of the ring $\mathcal{M}(X,\mathcal{A})$ and appropriately defined filters \emph{viz} $\mathcal{A}$-filters ($\mathcal{A}$-ultrafilters) on $X$. An $\mathcal{A}$-filter on $X$ is simply a filter whose members are $\mathcal{A}$-measurable sets. By exploiting this duality, we show that the set of all maximal ideals of $\mathcal{M}(X,\mathcal{A})$ endowed with the familiar hull-kernel topology, also called the structure space of $\mathcal{M}(X,\mathcal{A})$, is homeomorphic to the set $\hat{X}$ of all $\mathcal{A}$-ultrafilters on $X$, equipped with the Stone-topology (Theorem \ref{12}). This is the first important technical result in this article. This further yields a complete description of the maximal ideals of $\mathcal{M}(X,\mathcal{A})$ in terms of $\hat{X}$ (Theorem \ref{13}).  Incidentally, if $\mathcal{M}(X,\mathcal{A})$ is equipped with the $m$-topology, then all ideals in $\mathcal{M}(X,\mathcal{A})$ are closed (Theorem \ref{16}). 
We further note that the $\sigma$-algebra $\mathcal{A}$ on $X$ is finite when and only when each ideal (maximal ideal) of $\mathcal{M}(X,\mathcal{A})$ is fixed (Theorem \ref{15}). 

In Section~\ref{sec:residue}, we consider the order on the quotient ring of $\mathcal{M}(X,\mathcal{A})/I$ for an ideal $I$.
It turns out that $\mathcal{M}(X,\mathcal{A})/I$ is a lattice ordered ring with respect to the natural order induced by the order of the original ring $\mathcal{M}(X,\mathcal{A})$. 
We have the following characterization of the maximal ideals of $\mathcal{M}(X,\mathcal{A})$:
the ideal $I$ is maximal if and only if $\mathcal{M}(X,\mathcal{A})/I$ is totally ordered (Theorem \ref{19}). 
This is the main result in Section~\ref{sec:residue}. 
We define real and hyperreal maximal ideals of $\mathcal{M}(X,\mathcal{A})$ in an analogous manner to their counterparts in rings of continuous functions  and provide a characterization of those ideals in terms of the associated $\mathcal{A}$-ultrafilters on $X$ (Theorem \ref{23}). 

In Section~\ref{sec:intermediate}, we initiate the study of intermediate rings of measurable functions. 
By an intermediate ring of measurable functions we mean a subring $\mathcal{N}(X,\mathcal{A})$ of $\mathcal{M}(X,\mathcal{A})$ which contains $\mathcal{M}^*(X,\mathcal{A})$ of all the bounded measurable functions on $X$. 
The main technical tool in this section, which we borrow from the articles \cite{PSW}, \cite{RW1987}, \cite{RW1997}, is that of local invertibility of measurable functions on measurable sets in the given intermediate ring.  
With each maximal ideal $M$ in $\mathcal{N}(X,\mathcal{A})$, we associate an $\mathcal{A}$-ultrafilter $\mathcal{Z}_{\mathcal{N}}[M]$ on $X$ which leads to a bijection between the set of all maximal ideals of $\mathcal{N}(X,\mathcal{A})$ and the family of all $\mathcal{A}$-ultrafilters on $X$ (Theorems \ref{31} and \ref{32}). 
It is interesting to note that this bijective map becomes a homeomorphism provided the former set is equipped with the hull-kernel topology and the later with the Stone-topology (Theorem \ref{33}). 
This in essence says that the structure space of each intermediate ring of measurable functions is one and the same as that of the original ring $\mathcal{M}(X,\mathcal{A})$. 
In the concluding portion of Section~\ref{sec:intermediate}, we highlight a number of special properties which characterize $\mathcal{M}(X,\mathcal{A})$ among all the intermediate rings $\mathcal{N}(X,\mathcal{A})$ (Theorems \ref{34}, \ref{35}, and \ref{36}). 

In Section~\ref{sec:openproblems}, we highlight several properties enjoyed by the ring $\mathcal{M}(X,\mathcal{A})$ and the ring $C(Y)$ of all real-valued continuous functions defined over a $P$-space $Y$.  We conclude by raising a few questions about the relationship between rings of continuous functions on $P$-spaces and rings of measurable functions. 

\section{Ideals in $\mathcal{M}(X,\mathcal{A})$ versus $\mathcal{A}$-filters on $X$}
\label{sec:idealsfilters}

Throughout the paper, when we speak of ideal unmodified, we will always mean a proper ideal.
In this section, we introduce filters on the lattice of measurable sets, which we call $\mathcal{A}$-filters, and we show that each ideal (maximal ideal) of $\mathcal{M}(X,\mathcal{A})$ corresponds to an $\mathcal{A}$-filter ($\mathcal{A}$-ultrafilter) on $X$. 
We also describe the structure space of $\mathcal{M}(X,\mathcal{A})$.
\begin{definition}
A subfamily $\mathfrak{F}$ of $\mathcal{A}$ is called an $\mathcal{A}$-filter on $X$ if it excludes the empty set and is closed under finite intersection and formation of supersets from the family $\mathcal{A}$. An $\mathcal{A}$-filter on $X$ is said to be an $\mathcal{A}$-ultrafilter if it is not properly contained in any $\mathcal{A}$-filter on $X$. 
\end{definition}
By using Zorn's Lemma, it is easy to see that each $\mathcal{A}$-filter on $X$ extends to an $\mathcal{A}$-ultrafilter on $X$. Indeed $\mathcal{A}$-ultrafilters on $X$ are precisely those subfamilies of $\mathcal{A}$, which possess the finite intersection property and are maximal with respect to this property. Before formally initiating the duality between ideals in $\mathcal{M}(X,\mathcal{A})$ and the $\mathcal{A}$-filters on $X$, we write down the following well known technique of construction of measurable functions from smaller domains to larger ones.

\begin{theorem}[Pasting Lemma] (See \cite[Lemma 6]{AKA2000})\label{1}
If $\{ A_i\}_{i=1}^\infty$ is a countable family of members of $\mathcal{A}$ and $f: \cup_{i=1}^\infty A_i\mapsto \mathbb{R}$ is a function such that $f|_{A_i}$ is a measurable function for each $i$, then $f$ is also a measurable function. 
\end{theorem}

For any $f\in \mathcal{M}(X,\mathcal{A})$, we let $Z(f)$ denote the zero-set of $f$ and $c Z(f)=X\setminus Z(f)$ the co-zero set of $f$; here $Z(f)=\{x\in X: f(x)=0\}$. It is clear that zero-sets and co-zero sets of functions lying in $\mathcal{M}(X,\mathcal{A})$ are all members of $\mathcal{A}$. 
Conversely, each set $E\in \mathcal{A}$ is the zero set of some function in $\mathcal{M}(X,\mathcal{A})$, indeed $E=Z(\chi_{E^c})$, where $\chi_{E^c}$ is the characteristic function of $E^c=X\setminus E$ on $X$.

The following proposition is an immediate consequence of Pasting Lemma. 

\begin{theorem}\label{2}
For $f, g\in \mathcal{M}(X,\mathcal{A}), Z(f)\supseteq Z(g)$ if and only if $f$ is a multiple of $g$.
\end{theorem}

\begin{proof}
If $f$ is a multiple of $g$ in $\mathcal{M}(X,\mathcal{A})$, then it is trivial that $Z(f)\supseteq Z(g)$. Conversely let $Z(f)\supseteq Z(g)$. Define a function $h: (X,\mathcal{A})\mapsto \mathbb{R}$ by the following rule: $h(x)=\frac{f(x)}{g(x)}$ if $x\notin Z(g)$ and $h(x)=0$ if $x\in Z(g)$. Then by the Pasting Lemma $h$ is a member of $\mathcal{M}(X,\mathcal{A})$ and clearly $f=gh$.
\end{proof}

It follows from Theorem~\ref{2} that each $f\in \mathcal{M}(X,\mathcal{A})$ is a multiple of $f^2$, and hence $\mathcal{M}(X,\mathcal{A})$ is a Von-Neumann regular ring. 
It is well-known that any commutative reduced ring is Von-Neumann regular if and only if each of its prime ideals is maximal (see \cite[Theorem 1.16]{Goodearl}). 
Thus we have the following corollary. 
\begin{corollary}\label{3}
Every prime ideal of $\mathcal{M}(X,\mathcal{A})$ is maximal.
\end{corollary}
An ideal $I$ in a commutative ring $R$ with unity is called $z^\circ$-ideal if for each $a\in I, P_a\subseteq I$, where $P_a$ is the intersection of all minimal prime ideals containing $a$.
Since each ideal in a Von-Neumann regular ring is a $z^\circ$-ideal \cite[Remark 1.6(a)]{AKA2000}, it follows that all ideals of $\mathcal{M}(X,\mathcal{A})$ are $z^\circ$-ideals. 
This fact is also independently observed by \cite[Proposition 9]{EMY2018}.

For any ideal $I$ in $\mathcal{M}(X,\mathcal{A})$, let $Z[I]=\{Z(f):f\in I\}$, and for any $\mathcal{A}$-filter on $X$, let $Z^{-1}[\mathfrak{F}]=\{f\in \mathcal{M}(X,\mathcal{A}): Z(f)\in \mathfrak{F}\}$.
The following theorem entailing a duality between ideals in $\mathcal{M}(X,\mathcal{A})$ and $\mathcal{A}$-filters on $X$ is a measure-theoretic analog to \cite[Theorem 2.3]{GJ}, and can be established by using some routine arguments.  See also \cite[Proposition 3]{EMY2018}.

\begin{theorem}\label{6}
Let $I$ be an ideal in $\mathcal{M}(X,\mathcal{A})$, and $\mathfrak{F}$ be an $\mathcal{A}$-filter on $X$.
Then $Z[I]$ is an $\mathcal{A}$-filter on $X$, and $Z^{-1}[\mathfrak{F}]$ is an ideal in $\mathcal{M}(X,\mathcal{A})$
\end{theorem}

\begin{proposition}\label{5}
If $I$ is an ideal in $\mathcal{M}(X,\mathcal{A})$ containing a function $f$, then any $g$ in $\mathcal{M}(X,\mathcal{A})$ with $Z(g)=Z(f)$ is also a member of $I$.
\end{proposition}
The first of the following is a consequence of Proposition \ref{5} and the second directly from the definitions: if $I$ is an ideal of $\mathcal{M}(X,\mathcal{A})$ and $\mathfrak{F}$ is an $\mathcal{A}$-filter on $X$, then
\begin{equation}\label{eq:zzinv}
Z^{-1}Z[I]=I \qquad\text{and} \qquad ZZ^{-1}[\mathfrak{F}] = \mathfrak{F}.
\end{equation}
As a result, we have the following correspondence.
\begin{theorem}\label{thm:correspondence}
The map $Z: I\mapsto Z[I]$ is a bijective correspondence between ideals in $\mathcal{M}(X,\mathcal{A})$ and the $\mathcal{A}$-filters on $X$.
Moreover, if $M$ is a maximal ideal, then $Z[M]$ is an $\mathcal{A}$-ultrafilter, and if $\mathcal{U}$ is an $\mathcal{A}$-ultrafilter, then $Z^{-1}[\mathcal{U}]$ is a maximal ideal.
\end{theorem}
An ideal $I$ in $\mathcal{M}(X,\mathcal{A})$ is called fixed if $\cap Z[I]\neq \emptyset$, otherwise $I$ is called a free ideal. It was observed in \cite[Proposition 6]{EMY2018} by adapting the arguments in \cite[Theorem 4.6(a)]{GJ} that the complete list of fixed maximal ideals in $\mathcal{M}(X,\mathcal{A})$ is given by $\{M_p:p\in X\}$, where $M_p=\{ f\in \mathcal{M}(X,\mathcal{A}): f(p)=0\}$. If in addition, $\mathcal{A}$ separates points of $X$ in the sense that given any two distinct points $a, b$ in $X$, there is a member $E$ of $\mathcal{A}$, which contains exactly one of them, then $M_p\neq M_q$, whenever $p\neq q$ in $X$. It is established in \cite[Theorem 1.2]{M1971} that if a commutative ring $R$ with unity is also a Gelfand ring meaning that each prime ideal in $R$ extends to a unique maximal ideal, then the structure space of $R$ is Hausdorff. It follows therefore from Corollary \ref{3} that, the structure space of the ring $\mathcal{M}(X,\mathcal{A})$ is Hausdorff. 
It also follows from a more general result Theorem~\ref{32}, that we prove later in this paper.
Nevertheless, we shall produce an alternative proof of this assertion, by exploiting the duality between maximal ideals and $\mathcal{A}$-ultrafilters in Theorem \ref{thm:correspondence}.

\begin{theorem}\label{7}
The structure space of $\mathcal{M}(X,\mathcal{A})$ is a (compact) Hausdorff space.
\end{theorem} 

\begin{proof}
Let $M_1$ and $M_2$ be two distinct maximal ideals of $\mathcal{M}(X,\mathcal{A})$. Then by Theorem \ref{thm:correspondence}, the $\mathcal{A}$-ultrafilters $Z[M_1]$ and $Z[M_2]$ are also different, this implies in view of the maximality of an $\mathcal{A}$-ultrafilter on $X$ with respect to having the finite intersection property that, there exists $f_1\in M_1$ and $f_2\in M_2$ such that $Z(f_1)\cap Z(f_2)=\phi$. Let $g=\frac{(f_1)^2}{(f_1)^2+(f_2)^2}$. Then $g\in \mathcal{M}(X,\mathcal{A})$. Let $Z_1=\{x\in X: g(x)\leq \frac{1}{2}\}$ and $Z_2=\{ x\in X: g(x)\geq \frac{1}{2}\}$. then $Z_1$ and $Z_2$ are $\mathcal{A}$-measurable sets in $X$ and $Z_2\cup Z_2=X$. We can write $Z_1=Z(h_1)$ and $Z_2=Z(h_2)$, where $h_1, h_2\in \mathcal{M}(X,\mathcal{A})$. We see that $Z(h_2)\cap Z(f_1)=Z(h_1)\cap Z(f_2)=\phi $. Hence $h_2\notin M_1$ and $h_1\notin M_2$. Also $h_1h_2=0$; because $Z(h_1h_2)=Z_1\cup Z_2=X$. By \cite[Exercise 7M4]{GJ}, the structure space of $\mathcal{M}(X,\mathcal{A})$ is Hausdorff. 
\end{proof}

We now show that the hull-kernel topology of the structure space of $\mathcal{M}(X,\mathcal{A})$ can be identified with the Stone-topology on the set of all $\mathcal{A}$-ultrafilters on $X$. 
We now focus on a measure-theoretic analog of  \cite[Theorem 6.5]{GJ}.

Let $\hat{X}$ be an enlargement of the set $X$, with the intention that it will serve as an index set for the family of all $\mathcal{A}$-ultrafilers on $X$. For each $p\in \hat{X}$, let the corresponding $\mathcal{A}$-ultrafilter be denoted by $\mathcal{U}^p$ with the stipulation that for $p\in X$, $\mathcal{U}^p=\mathcal{U}_p=\{A\in \mathcal{A}: p\in A \}$. For each $A\in \mathcal{A}$, let $\bar{A}=\{ p\in \hat{X}: A\in \mathcal{U}^p\}$. Then $\{\bar{A}: A\in \mathcal{A}\}$ is a base for the closed sets of some topology, \emph{viz} the Stone-topology on $\hat{X}$. We shall simply write $\hat{X}$ to denote the set $\hat{X}$ with this Stone-topology. 
Observing that for all measurable sets $A,B\in \mathcal{A}$, $A\subseteq \bar{A}$, $A\subseteq B$ implies $\bar{A}\subseteq \bar{B}$, and $\bar{A}\cap X = A$, it is not hard to establish the following theorem which is a measure-theoretic analog of \cite[Theorem 6.5(b)]{GJ}.

\begin{theorem}\label{11}
For any $A\in \mathcal{A}$, $\bar{A}=cl_{\hat{X}}A$.
In particular $cl_{\hat{X}}X=\hat{X}$.
\end{theorem}

Given a maximal ideal $M$ in $\mathcal{M}(X,\mathcal{A})$, $Z[M]$ is an $\mathcal{A}$-ultrafilter on $X$, and therefore there exists a unique point $p\in \hat{X}$ such that $Z[M]=\mathcal{U}^p$. 
In this way, we obtain a map $\psi:\Max(\mathcal{M})\to \hat{X}$, where $\Max(\mathcal{M})$ is the set of all maximal ideals in $\mathcal{M}(X,\mathcal{A})$, such that $\psi(M)= p$.
By Theorem~\ref{thm:correspondence}, $\psi$ is a bijection.
Furthermore 
for any $f\in \mathcal{M}(X,\mathcal{A})$ and $M\in \Max(\mathcal{M})$, we have the following equivalence:
\begin{align*}
f\in M &\Leftrightarrow  Z(f) \in Z(M) && \text{(Theorem \ref{6})}\\
& \Leftrightarrow Z(f)\in \mathcal{U}^p && \text{where $\psi (M)=p$}\\
& \Leftrightarrow p\in cl_{\hat{X}}Z(f)
\end{align*}
Thus we can write $\psi(\mathcal{B}_f)=cl_{\hat{X}}Z(f) = \overline{Z(f)}$ for any $f\in \mathcal{M}(X,\mathcal{A})$, where $\mathcal{B}_f=\{ M\in \Max (\mathcal{M}): f\in M\}$. 
Therefore $\psi$ induces a bijection between the basic closed sets of the structure space $\Max (\mathcal{M})$ of $\mathcal{M}(X,\mathcal{A})$ and the basic closed sets $\bar{A}$ of the space $\hat{X}$ with the Stone-topology. 
Furthermore we observe that for $f\in \mathcal{M}(X,\mathcal{A}), \hat{X}\setminus cl_{\hat{X}}Z(f)=cl_{\hat{X}}(X\setminus Z(f))=cl_{\hat{X}}Z(g)$ for some $g\in \mathcal{M}(X,\mathcal{A})$. This shows that $cl_{\hat{X}}Z(f), f\in \mathcal{M}(X,\mathcal{A})$ are all clopen sets in the space $\hat{X}$. So, we can write:

\begin{theorem}\label{12}
The structure space $\Max (\mathcal{M})$ of $\mathcal{M}(X,\mathcal{A})$ is homeomorphic to the space $\hat{X}$, under the map $\psi: M\mapsto p$, where $Z[M]=\mathcal{U}^p$. Furthermore $\hat{X}$ is a compact Hausdorff zero-dimensional space.
\end{theorem}

Let us write for each $p\in \hat{X}$, $Z^{-1}[\mathcal{U}^p]=M^p$. Thus $\{ M^p: p\in \hat{X}\}$ is the complete list of maximal ideals of $\mathcal{M}(X,\mathcal{A})$. The following theorem is an analog of the Gelfand-Kalmogoroff theorem in rings of continuous functions for the maximal ideals of $\mathcal{M}(X,\mathcal{A})$.  It is a consequence of the arguments above.

\begin{theorem}\label{13}
	For each $p\in \hat{X}$, $M^p=\{ f\in \mathcal{M}(X,\mathcal{A}): p\in cl_{\hat{X}} Z(f)\}$. 
\end{theorem}

Our next goal is to characterize those measurable spaces $(X,\mathcal{A})$ for which the $\sigma$-algebras are finite in terms of the ideals of the ring $\mathcal{M}(X,\mathcal{A})$. 
But first we need the following subsidiary result.

\begin{lemma}\label{14}
Let $\mathcal{A}$ be an infinite $\sigma$-algebra on $X$. Then there exists a countably infinite family $\{E_n\}_{n=1}^\infty$ of pairwise disjoint nonempty members of $\mathcal{A}$.
\end{lemma}

\begin{proof}
An element $E\in \mathcal{A}$ is called an atom if $E\neq \emptyset$ and $E$ does not properly contain any nonempty member of $\mathcal{A}$. If there are infinitely many atoms of $\mathcal{A}$, then there is no more to prove because any two distinct atoms are pairwise disjoint. Assume therefore that there are only finitely many atoms of $\mathcal{A}$, say $A_1, A_2, A_3, \dotsc, A_n$. Let $A=\cup_{i=1}^n A_i$. We choose any nonempty set $B_0$ from $\mathcal{A}$, such that $B_0\cap A=\emptyset$. Since $B_0$ is not an atom, we can choose a nonempty set $B_1$ from $\mathcal{A}$ such that $B_1\subsetneq B_1$. We continue the process and having chosen $B_n$, let $B_{n+1}$ be a strictly smaller member of $\mathcal{A}\setminus \{\emptyset\}$, contained in $B_n$. In this way by induction, we construct a strictly decreasing chain $\{B_n\}_{n=0}^\infty$ of members of $\mathcal{A}$. Finally for each $n$, let $E_n=B_n\setminus B_{n+1}$. Then $\{E_n:n=0, 1, 2, \dotsc\}$ is a pairwise disjoint family of nonempty members of $\mathcal{A}$.
\end{proof}

In light of Lemma~\ref{14}, the notion of a $\sigma$-algebra being compact from \cite{EMY2018} (the collection of elements whose join is the top element has a finite subcollection whose join is the top element) is the equivalent to a $\sigma$-algebra being finite.
The following theorem is then an extension of \cite[Proposition 15]{EMY2018} with more equivalences and an alternative proof.
\begin{theorem}\label{15}
The statements written below are equivalent.

\begin{enumerate}
\renewcommand{\theenumi}{(\roman{enumi})}
\renewcommand{\labelenumi}{\theenumi}
\item \label{thm15item1} $\mathcal{M}(X,\mathcal{A})=\mathcal{M}^*(X,\mathcal{A})=\{f\in \mathcal{M}(X,\mathcal{A}): f \textit{ is bounded on } X\}$.
\item \label{thm15item2} Each ideal of $\mathcal{M}(X,\mathcal{A})$ is fixed.
\item \label{thm15item3} Each maximal ideal of $\mathcal{M}(X,\mathcal{A})$ is fixed.
\item \label{thm15item4} Each ideal of $\mathcal{M}^*(X,\mathcal{A})$ is fixed.
\item \label{thm15item5} Each maximal ideal of $\mathcal{M}^*(X,\mathcal{A})$ is fixed.
\item \label{thm15item6} $\mathcal{A}$ is a finite $\sigma$-algebra on $X$.
\end{enumerate}
\end{theorem}

\begin{proof}
$\ref{thm15item1}\Leftrightarrow\ref{thm15item6}$:
If $\ref{thm15item6}$ is false, then by Lemma \ref{14}, there exists a countably infinite family $\{E_n\}_{n=1}^\infty$ of pairwise disjoint nonempty sets in $\mathcal{A}$. The function $f:X\mapsto \mathbb{R}$, given by: $f(E_n)=n$ for $n\in \mathbb{N}$ and $f(X\setminus \cup_{n=1}^\infty E_n)=0$ is clearly an unbounded measurable function by the Pasting Lemma (Theorem \ref{1}). Thus $f\in \mathcal{M}(X,\mathcal{A})\setminus \mathcal{M}^*(X,\mathcal{A})$ and so $\ref{thm15item1}$ is false. 

Conversely, if $(i)$ is false, then there exists a $g\in \mathcal{M}(X,\mathcal{A})$ such that $g\geq 0$ and $g$ is unbounded above on $X$. Consequently there exists a countably infinite set of points $\{x_1, x_2, x_3, \dotsc\}$ in $X$ for which $f(x_1)< f(x_2)< f(x_3)<\dotsb<f(x_n)<\dotsb$. Let $F_n=\{x\in X: f(x)<f(x_{n+1})\}$. Then $F_1\subsetneq F_2\subsetneq \dotsb.$ is a strictly increasing sequence of nonempty members of $\mathcal{A}$. This renders \ref{thm15item6} false.  

$\ref{thm15item2}\Leftrightarrow\ref{thm15item6}$:
It is trivial that $\ref{thm15item6}\Leftrightarrow\ref{thm15item2}$. 
Conversely, if $\ref{thm15item6}$ is false, and $\{E_n\}_{n=1}^\infty$ is the guaranteed collection of pairwise disjoint nonempty sets, then 
\[
I=\{ f\in \mathcal{M}(X,\mathcal{A}): f(E_n)=0\text{ for all but finitely many $n$'s in $\mathbb{N}$}\}
\]
is a free ideal of $\mathcal{M}(X,\mathcal{A})$. Therefore the statement $\ref{thm15item2}$ is false. 

$\ref{thm15item2}\Leftrightarrow\ref{thm15item4}$: If $\ref{thm15item2}$ is true, then $\ref{thm15item4}$ follows from the equivalence of $\ref{thm15item1}$ and $\ref{thm15item2}$. 
Conversely, assume that $\ref{thm15item4}$ is true and $I$ is an ideal of $\mathcal{M}(X,\mathcal{A})$. then $I\cap \mathcal{M}^*(X,\mathcal{A})$ is an ideal of $\mathcal{M}^*(X,\mathcal{A})$ and is fixed. Now with each $f$ in $I$, we can associate a multiplicative unit 
\[
u_f = \frac{1}{1+\lvert f\rvert} 
\]
of $\mathcal{M}(X,\mathcal{A})$ such that $u_f \cdot f\in \mathcal{M}^*(X,\mathcal{A})$. 
This implies that $\cap_{f\in I}Z(f)=\cap_{f\in I}Z(u_f \cdot f)\supsetneq \cap_{g\in I\cap \mathcal{M}^*(X,\mathcal{A})} Z(g)\neq \phi$. This prove that $I$ is a fixed ideal of $\mathcal{M}(X,\mathcal{A})$.

Altogether the statements $\ref{thm15item1}$, $\ref{thm15item2}$, $\ref{thm15item4}$, and $\ref{thm15item6}$ are equivalent. 
The equivalence of $\ref{thm15item2}$ and $\ref{thm15item3}$ (respectively $\ref{thm15item4}$ and $\ref{thm15item5}$) follows from Zorn's Lemma. 
\end{proof}

\begin{definition}
For each $g$ in $\mathcal{M}(X,\mathcal{A})$ and each positive unit $u$ of this ring, set $m(g,u)=\{f\in \mathcal{M}(X,\mathcal{A}): \lvert f-g\rvert\leq u\}$. Then there exists a unique topology on $\mathcal{M}(X,\mathcal{A})$ which we call the $m$-topology in which for each $g$, $\{m(g,u): u \textit{ is a positive unit of } \mathcal{M}(X,\mathcal{A})\}$ is a neighbourhood base of it (compare this to \cite[Exercise 2N]{GJ}). 
\end{definition}
It is easy to prove that $\mathcal{M}(X,\mathcal{A})$ with the $m$-topology is a topological ring, by using some routine arguments and the fact that a continuous function of a real-valued measurable function is measurable. 
Furthermore it is not at all difficult to check that the set of all multiplicative units of the ring $\mathcal{M}(X,\mathcal{A})$ is an open set in this $m$-topology. It follows that if $I$ is a (proper) ideal of $\mathcal{M}(X,\mathcal{A})$, then its closure is also a (proper) ideal. 
Thus every maximal ideal is closed.

\begin{theorem}\label{16}
Each ideal in $\mathcal{M}(X,\mathcal{A})$ is closed in the $m$-topology.
\end{theorem}

\begin{proof}
By Theorem \ref{5}, we can write 
\begin{align*}
I&=\{f\in \mathcal{M}(X,\mathcal{A}): Z(f)\in Z[I]\}\\
&=\{f\in \mathcal{M}(X,\mathcal{A}): Z(f^n)\in Z[I]\}\\
&=\{f\in \mathcal{M}(X,\mathcal{A}): f^n\in I \textit{ for some } n\in \mathbb{N}\},
\end{align*}
Thus $I$ is the intersection of all prime ideals of $\mathcal{M}(X,\mathcal{A})$ containing it (see \cite[Corollary 0.18]{GJ}). 
As $\mathcal{M}(X,\mathcal{A})$ is Von Neumann regular, each of its prime ideals is maximal, and hence $I$ is the intersection of all maximal ideals containing it. 
As remarked in the comments preceeding the theorem, each maximal ideal of $\mathcal{M}(X,\mathcal{A})$ is closed;
hence $I$ is a closed subset of $\mathcal{M}(X,\mathcal{A})$.
\end{proof}

It was proved by Hewitt in \cite[Theorem 3]{Hewitt1948} that $C(X)$ with the $m$-topology is first-countable if and only if $X$ is pseudocompact.
We know give a characterization for $\mathcal{M}(X,\mathcal{A})$ with the $m$-topology to be first-countable.
\begin{theorem}
The $m$-topology of $\mathcal{M}(X,\mathcal{A})$ is first-countable if and only if $\mathcal{A}$ is finite.
\end{theorem}
\begin{proof}
Let $\mathcal{A}$ be finite.
Then by Theorem~\ref{15}, $\mathcal{M}(X,\mathcal{A}) = \mathcal{M}^*(X,\mathcal{A})$.
So $\mathcal{M}(X,\mathcal{A})$ is a Banach space with the sup norm, and it particular, it is metrizable, and hence its metric topology is first countable.
Furthermore, we observe that the $m$-topology is this norm topology, since if $u>0$ is a unit in $\mathcal{M}^*(X,\mathcal{A})$, then as $\mathcal{M}^*(X,\mathcal{A}) = \mathcal{M}(X,\mathcal{A})$, $1/u$ is in $\mathcal{M}^*(X,\mathcal{A})$, and so there is a $\lambda>0$, such that $u(x)>\lambda$ for all $x\in X$.
Then $m(f,u) \subseteq U(f,\lambda)$, where $U(f,\lambda) = \{g\in \mathcal{M}(x,\mathcal{A}\mid |f(x)-g(x)|\le \lambda\}$ is a closed base element of the norm topology.
Furthermore $U(f,\lambda) = m(f,\boldsymbol{\lambda})$, where $\boldsymbol{\lambda}(x) = \lambda$ for all $x\in X$.

Suppose instead that $\mathcal{A}$ is infinite.
Then by Theorem~\ref{14}, there is a countable family $\{A_n\}_{n\in \mathbb{N}}$ of pairwise disjoint non-empty sets $A_n$ from $\mathcal{A}$.
We claim that $\mathcal{M}(X,\mathcal{A})$ with the $m$-topology is not first-countable at the constant functions $\boldsymbol{0}$.
Suppose toward a contradiction, that $\boldsymbol{0}$ has a countable base $\{m(\boldsymbol{0},u_i)\}_{i\in \mathbb{N}}$, where $u_i(x)>0$ for all $x\in X$.
To obtain a contradiction, we construct a positive unit $u$ in $\mathcal{M}(X,\mathcal{A})$, such that $m(\boldsymbol{0},u_i)\not\subseteq m(\boldsymbol{0},u)$ for any $i\in \mathbb{N}$.
Indeed, let $u:X\to \mathbb{R}$ be defined as follows: 
\[
u(x) = \begin{cases}\frac{1}{2}u_n(x) & \text{if $x\in A_n$ for some $n\in \mathbb{N}$}\\
1 & \text{if $x\in (X-\bigcup_{n=1}^\infty A_n)$}\end{cases}
\]
By the pasting lemma (Lemma~\ref{1}), $u$ is measurable.
But for each $n$, $m(\boldsymbol{0},u_n)\not\subseteq m(\boldsymbol{0},u)$, since $\frac{2}{3}u_n\in m(\boldsymbol{0},u_n)$, but $\frac{2}{3}u_n\not\in m(\boldsymbol{0},u)$.
\end{proof}

\section{Residue class rings of $\mathcal{M}(X,\mathcal{A})$ modulo ideals}
\label{sec:residue}

In this section, we consider the ordering of a quotient ring of measurable functions by an absolutely convex ideal.
In what follows we denote $I(a)$ to be the residue class $I+a$ in $R/I$ which contains $a$.
Also, let $0$ be the the identity element $I$ of $R/I$.
An ideal $I$ of a lattice-ordered ring $R$ is called absolutely convex if whenever $|a|<|b|$ and $b\in I$ then $a\in I$.
We begin by recalling the following 
well-known results (see \cite[\S 5.3]{GJ}). 
\begin{proposition}\label{thm:absConvOrdering}
If $I$ is an absolutely convex ideal in a lattice ordered ring $R$, then 
\begin{enumerate}
\item $R/I$ is a lattice ordered ring, using the following ordering: $I(a) \ge 0$ if there exists an $x\in R$ such that $x\ge 0$ and $a\equiv x\pmod{I}$.
\item $I(a )\ge 0$ if and only if $a \equiv |a|\pmod{I}$
\item $I(|a|) = |I(a)|$ for each $a\in R$.
\end{enumerate}
\end{proposition}

Note that $\mathcal{M}(X,\mathcal{A})$ is a lattice-ordered ring with the natural order for each $f,g\in \mathcal{M}(X,\mathcal{A})$,
$f\le g$ if and only if $f(x)\le g(x)$ for all $x\in X$.

\begin{proposition}\label{thm:convex}
Each ideal in $\mathcal{M}(X,\mathcal{A})$ is absolutely convex.
\end{proposition}
\begin{proof}
If $f,g\in \mathcal{M}(X,\mathcal{A})$, $g\in I$, and $\lvert f\rvert \leq \lvert g\rvert $, then $Z(g)\subseteq Z(f)$; 
consequently by Theorem \ref{2}, $f$ is a multiple of $g$, hence $f\in I$.
\end{proof} 
The following theorem follows immediately from Propositions~\ref{thm:absConvOrdering} and \ref{thm:convex}.
\begin{theorem}\label{17}
If $I$ is an ideal of $\mathcal{M}(X,\mathcal{A})$, then the quotient ring $\mathcal{M}(X,\mathcal{A})/I$ is a lattice ordered ring.
\end{theorem}

The following theorem provides useful description of non-negative elements of the quotient ring $\mathcal{M}(X,\mathcal{A})/I$.  Compare with \cite[\S5.4(a)]{GJ}.
\begin{theorem}\label{18}
Let $I$ be an ideal of $\mathcal{M}(X,\mathcal{A})$ and $f\in \mathcal{M}(X,\mathcal{A})$. Then $I(f)\geq 0$ if and only if there exists $E\in Z[I]$ such that $f\geq 0$ on $E$. 
\end{theorem}
\begin{proof}
Let $I(f)\ge 0$.
Then it follows from Proposition~\ref{thm:absConvOrdering} that $I(f) =  |I(f)| = I(|f|)$.
Consequently $f-|f|\in I$, and hence $E = Z(f-|f|) \in Z[I]$.
It is clear that $f\ge 0$ on $E$.

Conversely, if $E\in Z[I]$ and $f\ge 0$ on $E$, then it is clear that $E\subseteq Z(f-|f|)$.
Since $Z[I]$ is a $\mathcal{A}$-filter on $X$, it follows that $Z(f-|f|)\in Z[I]$.
Hence by Proposition~\ref{5}, we can write $f-|f|\in I$.
Since $|f|\ge 0$, we have that $I(|f|)\ge 0$.
But by Proposition~\ref{thm:absConvOrdering}, $I(|f|) = I(f)$.
So $I(f)\ge 0$.
\end{proof}
The following result gives a characterization of maximal ideals of $\mathcal{M}(X,\mathcal{A})$.

\begin{theorem}\label{19}
For an ideal $I$ in $\mathcal{M}(X,\mathcal{A})$, the following statements are equivalent:
\begin{enumerate}
\renewcommand{\theenumi}{(\roman{enumi})}
\renewcommand{\labelenumi}{\theenumi}
    \item \label{thm19item1} The ideal $I$ is a maximal ideal of $\mathcal{M}(X,\mathcal{A})$.
    \item \label{thm19item2} Given $f\in \mathcal{M}(X,\mathcal{A})$, there exists $E\in Z[I]$ on which $f$ does not change its sign.
    \item \label{thm19item3} The residue class ring $\mathcal{M}(X,\mathcal{A})/I$ is totally ordered.
\end{enumerate}
\end{theorem}

\begin{proof}
$\ref{thm19item1} \Rightarrow \ref{thm19item2}$: Suppose $\ref{thm19item1}$ holds and let $f\in \mathcal{M}(X,\mathcal{A})$. Then since $(f\vee 0)(f\wedge g)=0$ and each maximal ideal is prime, it follows that $f\vee 0\in I$ or $f\wedge 0\in I$. Consequently $Z(f\vee 0)\in Z[I]$ or $Z(f\wedge 0)\in Z[I]$. We note that $f\leq 0$ on $Z(f\wedge 0)$ and $f\geq 0$ on $Z(f\vee 0)$.

$\ref{thm19item2} \Rightarrow \ref{thm19item3}$: Suppose $\ref{thm19item2}$ is true.  Let $f\in \mathcal{M}(X,\mathcal{A})$. Then there is an $E\in Z[I]$ on which $f\geq 0$ or $f\leq 0$. This implies in view of Theorem \ref{18} that $I(f)\geq 0$ or $I(f)\leq 0$ in $\mathcal{M}(X,\mathcal{A})/I$. Thus $\mathcal{M}(X,\mathcal{A})/I$ is totally ordered.

$\ref{thm19item3}\Rightarrow \ref{thm19item1}$:  Suppose $\ref{thm19item3}$ is true. Let $g, h\in \mathcal{M}(X,\mathcal{A})$ such that $gh\in I$. By the condition $\ref{thm19item3}$, we can write either $I(\lvert g\rvert -\lvert h\rvert )\geq 0$ or $I(\lvert g\rvert -\lvert h\rvert )\leq 0$.  Without loss of generality, $I(\lvert g\rvert -\lvert h\rvert )\geq 0$. It follows from Theorem \ref{18} that there is an $E\in Z[I]$ such that $\lvert g\rvert -\lvert h\rvert \geq 0$ on $E$. This implies that $E\cap Z(g)\subseteq Z(h)$ and hence $E\cap Z(gh)\subseteq Z(h)$. As $E\in Z[I]$ and $gh\in I$, the last relation implies that $Z(h)\in Z[I]$, hence $h\in I$ by Proposition \ref{5}. If we assume that $I(\lvert g\rvert -\lvert h\rvert )\leq 0$, we could have obtained analogously that $g\in I$. Thus either $g\in I$ or $h\in I$. Hence $I$ is a prime ideal and therefore maximal ideal in $\mathcal{M}(X,\mathcal{A})$.
\end{proof}

\begin{definition}
 A totally ordered field $F$ is called archimedian if given $\alpha\in F$, there is an $n\in \mathcal{N}$ such that $n> \alpha$.  Otherwise $F$ is called non-archimedian. 
\end{definition}
So, if $F$ is non archimedian, then there is an element $\alpha\in F$ such that $\alpha >n$ for each $n\in \mathbb{N}$. Such an $\alpha$ is called an infinitely large member of $F$. The reciprocal of an infinitely large member is called an infinitely small member of $F$. Thus a non archimedian totally ordered field is characterized by the presence of infinitely large (equivalently infinitely small) members in it. 

If $M$ is a maximal ideal of $\mathcal{M}(X,\mathcal{A})$ and $\pi:\mathcal{M}(X,\mathcal{A})\to \mathcal{M}(X,\mathcal{A})/M$ given by $f\mapsto M(f)$ is the canonical map, then $M$ and $\mathcal{M}(X,\mathcal{A})/M$ are called real if the set of images of the constant functions under $\pi$ is all of $\mathcal{M}(X,\mathcal{A})/M$ (in which case, $\mathcal{M}(X,\mathcal{A})/M$ is isomorphic to $\mathbb{R}$), and hyperreal otherwise. 
The residue class field $\mathcal{M}(X,\mathcal{A})/M$ is archimedian if and only if it is real, since 
an ordered field is archimedian if and only if it is isomorphic to  subfield of $\mathbb{R}$ (\cite[\S0.21]{GJ}), and 
identity is the only non-zero homomorphism of $\mathbb{R}$ into itself (\cite[\S0.22]{GJ}).

The following result relates infinitely large members in the residue class fields of $\mathcal{M}(X,\mathcal{A})$ modulo hyperreal maximal ideals of $\mathcal{M}(X,\mathcal{A})$ and unbounded functions in $\mathcal{M}(X,\mathcal{A})$.

\begin{theorem}\label{20}
 For a given maximal ideal $M$ in $\mathcal{M}(X,\mathcal{A})$ and an $f\in \mathcal{M}(X,\mathcal{A})$, the following statements are equivalent:
\begin{enumerate}
\renewcommand{\theenumi}{(\roman{enumi})}
\renewcommand{\labelenumi}{\theenumi}
\item \label{thm20item1} $\lvert M(f)\rvert $ is an infinitely large member of the residue class field $\mathcal{M}(X,\mathcal{A})/M$.
\item \label{thm20item2} $f$ is unbounded on every set in the $\mathcal{A}$-ultrafilter $Z[M]$.
\item \label{thm20item3} For each $n\in \mathbb{N}$, $E_n=\{x\in X: \lvert f(x)\rvert \geq n\}\in Z[M]$.  
\end{enumerate}
\end{theorem}

\begin{proof}
$\ref{thm20item1} \Leftrightarrow \ref{thm20item2}$ are equivalent because $\lvert M(f)\rvert$ is not infinitely large means there is an $n\in \mathbb{N}$ such that $\lvert M(f)\rvert \leq M(n)$ ($n$ stands for the constant function with value $n$ on $X$).
By Theorem \ref{18}, this is the case when and only when $\lvert f\rvert \leq n$ on some $E\in Z[M]$.

$\ref{thm20item3}\Rightarrow \ref{thm20item1}$ is a straightforward consequence of Theorem \ref{18}.

$\ref{thm20item1} \Rightarrow \ref{thm20item3}$: If $\ref{thm20item1}$ hold, then $\lvert M(f)\rvert \geq M(n), \forall n\in \mathbb{N}$, i.e.\ $M(\lvert f\rvert)\geq M(n)$ for all $n\in \mathbb{N}$. It follows from Theorem \ref{18} that there is an $E\in Z[M]$ for which $\lvert f\rvert \geq n$ on $E$.
Such an $E$ is contained in $E_n$, and hence $E_n\in Z[M]$. 
\end{proof}

\begin{theorem}\label{21}
An $f\in \mathcal{M}(X,\mathcal{A})$ is unbounded on $X$ if and only if there is a maximal ideal $M$ in $\mathcal{M}(X,\mathcal{A})$ for which $\lvert M(f)\rvert$ is infinitely large in $\mathcal{M}(X,\mathcal{M})/M$.
\end{theorem}

\begin{proof}
If $f$ is unbounded on $X$, then $E_n=\{ x\in X: \lvert f(n)\rvert \geq n\}\neq \emptyset$ for each $n\in \mathbb{N}$.
Therefore $\{E_n: n\in \mathbb{N}\}$ is a family of $\mathcal{A}$-measurable sets with the finite intersection property and is therefore extendable to an $\mathcal{A}$-ultrafilter $\mathcal{U}$ on $X$. Clearly $\mathcal{U}=Z[M]$ for a unique maximal ideal $M$ in $\mathcal{M}(X,\mathcal{A})$. Thus $E_n\in Z[M]$ for every $n\in \mathbb{N}$.
Hence by Theorem \ref{20}, $\lvert M(f)\rvert $ is infinitely large.

Conversely if $\lvert M(f)\rvert$ is infinitely large, then by Theorem \ref{20}, $f$ is unbounded on every members in $Z[M]$; in particular, $f$ is unbounded on $X$.
\end{proof}

The following theorem is a measure-theoretic analog of \cite[Theorem 5.14]{GJ}. 
\begin{theorem}\label{23}
For any maximal ideal $M$ in $\mathcal{M}(X,\mathcal{A})$, the following statements are equivalent:
\begin{enumerate}
\renewcommand{\theenumi}{(\roman{enumi})}
\renewcommand{\labelenumi}{\theenumi}
\item $M$ is a real maximal ideal.
\item $Z[M]$ is closed under countable intersection.
\item $Z[M]$ has the countable intersection property.
\end{enumerate}
\end{theorem}
\begin{proof}
((i)$\Rightarrow$(ii))
Suppose (ii) is false.
This means that there is a countable collection of functions $(f_n)_{n\in \mathbb{N}}$ in $M$ such that $\bigcap Z(f_n) \not\in Z[M]$.
Then $g$, defined by $g(x) = \sum_{n\in \mathbb{N}} |f_n(x)| \wedge 3^{-n}$ for each $x\in X$, is a member of $\mathcal{M}(X,\mathcal{A})$, since whenever a series of real-valued measurable functions is uniformly convergent, then its limit function is measurable and also real-valued.
Since $g\ge 0$ by construction, it follows that $M(g)\ge 0$.
But $Z(g) = \bigcap Z(f_n)\not\in Z[M]$.
So $g\not\in M$ making $M(g)$ strictly positive.
For each $k\in \mathbb{N}$, $Z(f_1)\cap Z(f_2)\cap \dotsb\cap Z(f_k)\in Z[M]$, and on this set, $g\le \sum_{n=k}^\infty 3^{-n} = 2^{-1}3^{-k}$.
Consequently, by Theorem~\ref{18}, it follows that $M(g) \le M(2^{-1}3^{-k})$ for each $k\in \mathbb{N}$.
Thus $M(g)$ is infinitely small, and $M$ is not real.

((ii)$\Rightarrow$(iii)) is trivial since $\emptyset\not\in Z[M]$.

((iii)$\Rightarrow$(i)) Suppose (i) is false, that is $M$ is hyperreal.
Then $\mathcal{M}(X,\mathcal{A})/M$ is non-archimedean.  
Consequently, there exists $f\in M$, such that $f\ge 0$ and $M(f)$ is infinitely large.
Hence by Theorem~\ref{20}, for each $n\in \mathbb{N}$, $E_n=\{x\in X: \lvert f(x)\rvert \geq n\}\in Z[M]$.
But then $\bigcap Z[M] \subseteq \bigcap E_n = \emptyset$, in which case (iii) is false.
\end{proof}

The following example gives a measurable space with infinite $\sigma$-algebra for which every real maximal ideal is fixed.
\begin{example}\label{ex:fixedrealLebesgue}
Let $\mathcal{A}$ be the $\sigma$-algebra of all Lebesgue measurable sets in $\mathbb{R}$. Then $(\mathbb{R}, \mathcal{A})$ is a measurable space. Let $M$ be a real maximal ideal of $\mathcal{M}(X,\mathcal{A})$.
Let $i\in \mathcal{M}(X,\mathcal{A})$ be the constant function. 
Because $M$ is real, there is an $r\in \mathbb{R}$ such that $M(i)=M(r)$ and hence $Z(i-r)\in Z[M]$. 
But $Z(i-r)$ is a one-point set. 
So $Z[M]$ is fixed, i.e. M is fixed. 
\end{example}

In light of Theorem \ref{23} and Example~\ref{ex:fixedrealLebesgue}, 
a maximal ideal in $\mathcal{M}(\mathbb{R}, \mathcal{A})$ is real if and only if it is fixed.
We will see in Example \ref{ex:realmaxnotfixed} that there exists a measurable space $(X,\mathcal{A})$ such that $\mathcal{M}(X,\mathcal{A})$ has a real free maximal ideal.

\begin{definition}
An $\mathcal{A}$-ultrafilter $\mathcal{U}$ on $X$ is a real $\mathcal{A}$-ultrafilter if it is closed under countable intersection (or equivalently which has countable intersection property).
\end{definition}

Example \ref{ex:fixedrealLebesgue} raises the following questions.
\begin{question}\label{quest:realfixed}
Can we characterize the measurable spaces $(X,\mathcal{A})$ for which each real $\mathcal{A}$-ultrafilter on $X$
 is fixed? 
\end{question}

\begin{question}
If $X$ is a real compact space and $\mathcal{B}(X)$ is the $\sigma $-algebra of all Borel subsets of $X$, does the measure space $(X,\mathcal{B}(X))$ satisfy the property that each real $\mathcal{B}(X)$-ultrafilter on $X$ is fixed? 
\end{question}

\section{ Ideals in intermediate rings of measurable functions}
\label{sec:intermediate}
By an intermediate ring (of measurable functions), we mean any ring $\mathcal{N}(X,\mathcal{A})$ lying between $\mathcal{M}^*(X,\mathcal{A})$ and $\mathcal{M}(X,\mathcal{A})$. Let $\Omega (X,\mathcal{A})$ stand for the aggregate of all these intermediate rings.

\begin{definition}
	Let $E\in \mathcal{A}$. We say that $f\in \mathcal{N}(X,\mathcal{A})$ is $E$-regular if there exist $g\in \mathcal{N}(X,\mathcal{A})$ such that $fg|_{E^c} =1$, where $E^c=X\setminus E$.
\end{definition}

It is clear that $f$ is $E$-regular if and only if $f^2$ is $E$-regular if and only if $\lvert f\rvert$ is $E$-regular.

\begin{definition}
For $f\in \mathcal{N}(X,\mathcal{A})$ define 
\[
\mathcal{Z}_{\mathcal{N}}(f)=\{ E\in \mathcal{A}: f \textit{ is } E^c\textit{-regular}\}
\]
and for any $S\subseteq \mathcal{N}(X,\mathcal{A})$ and any $\mathfrak{F}\subseteq \mathcal{A}$, let 
\begin{align*}
\mathcal{Z}_{\mathcal{N}}[S]=\bigcup_{f\in S}\mathcal{Z}_{\mathcal{N}}(f)\qquad\text{and}\qquad
\mathcal{Z}^{-1}_{\mathcal{N}}[\mathfrak{F}]=\{f\in \mathcal{M}(X,\mathcal{A}): \mathcal{Z}_{\mathcal{N}}(f)\subseteq \mathfrak{F}\}.
\end{align*}
\end{definition} 

The following facts are measure theoretic analogs of results in \cite{PSW} and \cite[Lemma 3.1]{SW2015}.
We omit proofs as they are straightforward.

\begin{theorem}\label{24}
Let $\mathcal{N}(X,\mathcal{A})$ be an intermediate ring of measurable functions.
\begin{enumerate}
\renewcommand{\theenumi}{(\roman{enumi})}
\renewcommand{\labelenumi}{\theenumi}
\item \label{thm24item1} If $I$ is an ideal in $\mathcal{N}(X,\mathcal{A})$, then $\mathcal{Z}_{\mathcal{N}}[I]$ is an $\mathcal{A}$-filter on $X$.
\item \label{thm24item2} For any $\mathcal{A}$-filter $\mathfrak{F}$ on $X$, $I=\mathcal{Z}^{-1}_{\mathcal{N}}[\mathfrak{F}]$ is an ideal in $\mathcal{N}(X,\mathcal{A})$.
\item \label{thm24item3} For $f\in \mathcal{N}(X,\mathcal{A})$, $\cap  \mathcal{Z}_{\mathcal{N}}(f)=Z(f)$.
\end{enumerate}
\end{theorem}

For any measurable set $E$, let $\langle E\rangle$ be the principal $\mathcal{A}$-filter whose intersection is $E$.
\begin{lemma}\label{27}
	If $E\in \mathcal{A}$, then there exists $f\in \mathcal{N}(X,\mathcal{A})$ such that $E=Z(f)$ and $\mathcal{Z}_{\mathcal{N}}(f)=\langle Z(f)\rangle$.
\end{lemma}

\begin{proof}
	Take $f=\chi_{E^c}$. Then $Z(f)=E$ and surely $f$ is invertible on $E^c$.
	This means that $E\in \mathcal{Z}_\mathcal{N}(f)$. Hence $\langle E\rangle \subseteq \mathcal{Z}_\mathcal{N}(f)$. Conversely if $F\in \mathcal{Z}_\mathcal{N}(f)$, then $F\supseteq Z(f)$ by Theorem \ref{24}\ref{thm24item3}, which implies $F\in \langle E\rangle$. 
	Thus $\mathcal{Z}_\mathcal{N}(f)\subseteq \langle E\rangle$.
\end{proof}

It is easy to see that for any ideal $I$ in $\mathcal{N}(X,\mathcal{A})$, and any $\mathcal{A}$-filter $\mathfrak{F}$ on $X$, 
\begin{equation}\label{eq:scriptZscriptZinv}
\mathcal{Z}_{\mathcal{N}}^{-1}[\mathcal{Z}_{\mathcal{N}}[I]] \supseteq I\qquad\text{and}\qquad\mathcal{Z}_\mathcal{N}\mathcal{Z}_\mathcal{N}^{-1}[\mathfrak{F}]\subseteq \mathfrak{F}.
\end{equation}
Compare this with \eqref{eq:zzinv} which gives equality when $\mathcal{N}(X,\mathcal{A})$ include all measurable functions.

In an intermediate ring of continuous functions $A(X)$, if $M$ is a maximal ideal, then $\mathcal{Z}_A(M)$ need not be a $z$-ultrafilter on $X$, or even a prime $z$-filter (see \cite[p.~154]{RW1997}).
With intermediate rings of measurable functions, we have the following.
\begin{theorem}\label{28}
Let $\mathcal{N}(X,\mathcal{A})$ be an intermediate ring of measurable functions. Then
\begin{enumerate}
\renewcommand{\theenumi}{(\roman{enumi})}
\renewcommand{\labelenumi}{\theenumi}
\item \label{thm28item1} If $M$ is a maximal ideal in $\mathcal{N}(X,\mathcal{A})$, then $\mathcal{Z}_{\mathcal{N}}[M]$ is an $\mathcal{A}$-ultrafilter on $X$.

\item \label{thm28item2} If $\mathcal{U}$ is an $\mathcal{A}$-ultrafilter on $X$, then $\mathcal{Z}^{-1}_\mathcal{N}[\mathcal{U}]$ is a maximal ideal in $\mathcal{N}(X,\mathcal{A})$.
\end{enumerate}
\end{theorem}

\begin{proof}
\ref{thm28item1}:
Let $M$ be a maximal ideal in $\mathcal{N}(X,\mathcal{A})$.
Then $\mathcal{Z}_{\mathcal{N}}[M]$ is an $\mathcal{A}$-filter on $X$ (by Theorem \ref{24}\ref{thm24item1}). Hence there exists an $\mathcal{A}$-ultrafilter $\mathcal{U}$ on $X$ such that $\mathcal{Z}_{\mathcal{N}}[M]\subseteq \mathcal{U}$. We claim that $\mathcal{Z}_\mathcal{N}[M]=\mathcal{U}$. So let us choose $E\in \mathcal{U}$. Then by Lemma \ref{27}, we can find an $f\in \mathcal{N}(X,\mathcal{A})$ such that $\mathcal{Z}_{\mathcal{N}}(f)=\langle E\rangle =\langle Z(f)\rangle$. By \eqref{eq:scriptZscriptZinv}, we can write $M\subseteq \mathcal{Z}_{\mathcal{N}}^{-1}[\mathcal{Z}_{\mathcal{N}}[M]]\subseteq \mathcal{Z}_{\mathcal{N}}^{-1}[\mathcal{U}]$, and by Theorem \ref{24}\ref{thm24item2}, $\mathcal{Z}_\mathcal{N}^{-1}[\mathcal{U}]$ is an ideal in $\mathcal{N}(X,\mathcal{A})$.  
This implies, in view of the maximality of $M$ and also the properness of the ideal $\mathcal{Z}_\mathcal{N}^{-1}[\mathcal{U}]$, that $M=\mathcal{Z}_{\mathcal{N}}^{-1}[\mathcal{U}]$.  Since $\mathcal{Z}_\mathcal{N}(f)\subseteq \mathcal{U}$, it follows that $f\in \mathcal{Z}_{\mathcal{N}}^{-1}[\mathcal{U}]=M$, and hence $\mathcal{Z}_\mathcal{N}(f)\subseteq \mathcal{Z}_\mathcal{N}[M]$.  Since $\mathcal{Z}_{\mathcal{N}}(f) = \langle E \rangle$, we then have that 
$E\in \mathcal{Z}_\mathcal{N}[M]$. Thus $\mathcal{U}\subseteq \mathcal{Z}_\mathcal{N}[M]$. Hence $\mathcal{Z}_{\mathcal{N}}[M]=\mathcal{U}$.

\ref{thm28item2}: By Theorem \ref{thm:correspondence}, $\mathcal{U}=Z[M']$ for some maximal ideal $M'$ in $\mathcal{M}(X,\mathcal{A})$. It is clear that $M'\cap \mathcal{N}(X,\mathcal{A})$ is a prime ideal in $\mathcal{N}(X,\mathcal{A})$ which is extendable to a maximal ideal $M$ of $\mathcal{N}(X,\mathcal{A})$. Enough to prove that $\mathcal{Z}_{\mathcal{N}}^{-1}[\mathcal{U}]\supseteq M$ (and this implies in view of the maximality of $M$ that $\mathcal{Z}_{\mathcal{N}}^{-1}[\mathcal{U}]= M$). We argue by contradiction and assume that there exists an $f\in M$ such that $f\notin \mathcal{Z}_{\mathcal{N}}^{-1}[\mathcal{U}]$; this means that $\mathcal{Z}_{\mathcal{N}}(f)$ is not contained in $\mathcal{U}$. Consequently there exists $E\in \mathcal{Z}_{\mathcal{N}}(f)$ such that $E\notin \mathcal{U}$. Now $E\in \mathcal{Z}_{\mathcal{N}}(f)$ means that there exists $h\in \mathcal{N}(X,\mathcal{A})$ such that $fh|_{E^c}=1$. On the other hand $E\notin \mathcal{U}$ implies that there exists $k\in M'$ such that $Z(k)\cap E=\emptyset$ and without loss of generality, we can assume that $k$ is bounded on $X$ and therefore $k\in M'\cap \mathcal{N}(X,\mathcal{A})$ and hence $k\in M$. Let $l=\chi_{E}$. Then $Z(l)\supseteq Z(k)$ implies by Theorem \ref{2} that $l$ is a multiple of $k$ in the ring $\mathcal{M}(X,\mathcal{A})$. Since $k\in M'$, which is a maximal ideal in $\mathcal{M}(X,\mathcal{A})$, it follows that $l\in M'$. Also $l$ is bounded on $X$, hence $l\in \mathcal{N}(X,\mathcal{A})$; therefore $l\in M'\cap \mathcal{N}(X,\mathcal{A})$, and consequently $l\in M$. Since $f\in M$, this implies that $f^2h^2+l\in M$ (taking care of the fact that $h\in \mathcal{N}(X,\mathcal{A})$). Finally note that $f^2h^2+l\geq 1$ on $(X,\mathcal{A})$, and therefore $f^2h^2+l\geq 1$ is bounded away from zero on $X$ and hence $f^2h^2+l\geq 1$ is a unit of $\mathcal{N}(X,\mathcal{A})$ --- this is a contradiction.
\end{proof}

Since, we have already established that if $M$ is a maximal ideal of $\mathcal{N}(X,\mathcal{A})$, then $\mathcal{Z}_{\mathcal{N}}[M]$ is an $\mathcal{A}$-ultrafilter on $X$, the following fact is immediate:

\begin{theorem}\label{31}
 The map $\mathcal{Z}_\mathcal{N}: \Max(\mathcal{N})\mapsto  \hat{X}$ described by $M\mapsto \mathcal{Z}_{\mathcal{N}}[M]$ is a bijection from the set of all maximal ideals of $\mathcal{N}(X,\mathcal{A})$ onto the set of all $\mathcal{A}$-ultrafilters on $X$.
\end{theorem}

\begin{theorem}\label{32}
 The structure space $\Max(\mathcal{N})$ of the intermediate ring $\mathcal{N}(X,\mathcal{A})$ is a (compact) Hausdorff space.
\end{theorem}

\begin{proof}
The proof is analogous to that of \cite[Theorem 3.6]{RW1997}.
Let $M_1, M_2\in \Max(\mathcal{N})$ with $M_1\neq M_2$. 
By \cite[Exercise 7M4]{GJ}, it suffices to find $h_1, h_2\in \mathcal{N}(X,\mathcal{A})$ such that $h_1\notin M_1$, $h_2\notin M_2$, and $h_1h_2=0$, as $\{0 \}=\cap \Max(\mathcal{N})$. 
We observe that there exist $E_1\in \mathcal{Z}_{\mathcal{N}}[M_1]$ and $E_2\in \mathcal{Z}_{\mathcal{N}}[M_2]$ such that $E_1\cap E_2=\emptyset$, for otherwise $\mathcal{Z}_\mathcal{N}[M_1]\cup \mathcal{Z}_\mathcal{N}[M_2]$ is a family of members of $\mathcal{A}$ with finite intersection property and hence there is an $\mathcal{A}$-filter $\mathcal{F}$ on $X$ such that $\mathcal{Z}_\mathcal{N}[M_1]\cup \mathcal{Z}_\mathcal{N}[M_2]\subseteq \mathcal{F}$; consequently $M_1\cup M_2\subseteq \mathcal{Z}_{\mathcal{N}}^{-1}[\mathcal{F}]$, which is a proper ideal of $\mathcal{N}(X,\mathcal{A})$, contradicting the maximality and distinctness of $M_1$ and $M_2$. 
Now since $E_1\in \mathcal{Z}_{\mathcal{N}}[M_1]$ and $E_2\in \mathcal{Z}_{\mathcal{N}}[M_2]$ with $E_1\cap E_2= \emptyset$, there exists $f\in M_1$ and $g\in M_2$ such that $E_1\in \mathcal{Z}_{\mathcal{N}}(f)$ and $E_2\in \mathcal{Z}_{\mathcal{N}}(g)$. This means that, there exist $f_1, g_1\in \mathcal{N}(X,\mathcal{A})$ such that $ff_1|_{E^c_1}=1$ and $ff_2|_{E^c_2}=1$. Since $ff_1\in M_1$ and $gg_1\in M_2$, it follows that $1-ff_1\notin M_1$ and $1-gg_1\notin M_2$. But we note that $(1-ff_1)(1-gg_1)=0$.
\end{proof} 

\begin{theorem}\label{33}
Let $\mathcal{N}(X,\mathcal{A})$ be an intermediate ring.
Then the bijection $\mathcal{Z}_\mathcal{N}: \Max(\mathcal{N})\mapsto \hat{X}$ by $M\mapsto \mathcal{Z}_{\mathcal{N}}[M]$ is a homeomorphism. 
\end{theorem}

\begin{proof}
Since each of the two spaces $\Max(\mathcal{N})$ and $\hat{X}$ is already known to be a compact Hausdorff space, it suffices to check that $\mathcal{Z}_{\mathcal{N}}$ is a closed map. A typical basic closed set in the space $\Max(\mathcal{N})$ is a set of the form $N_f=\{ M\in \Max(\mathcal{N}): f\in M \}$, for some $f\in \mathcal{N}(X,\mathcal{A})$. It is enough to show that $\mathcal{Z}_{\mathcal{N}}(N_f)=\cap \{\bar{E}: E\in \mathcal{Z}_\mathcal{N}(f) \}$, which is an intersection of a family of basic closed sets in $\hat{X}$, and is hence a closed set in $\hat{X}$. We see that if $M\in N_f$, then $f\in M$.
Consequently $\mathcal{Z}_\mathcal{N}(f)\subseteq \mathcal{Z}_\mathcal{N}[M]$, meaning if $E\in \mathcal{Z}_\mathcal{N}(f)$, then $E$ belongs to the $A$-ultrafilter $\mathcal{Z}_\mathcal{N}[M]$, so that $\mathcal{Z}_\mathcal{N}[M]\in \bar{E}$. Thus $\mathcal{Z}_\mathcal{N}(N_f)\subseteq \cap \{\bar{E}: E\in \mathcal{Z}_\mathcal{N}(f) \}$. 
Conversely if $M\in \Max(\mathcal{N})$ such that $\mathcal{Z}_\mathcal{N}[M]\in \bar{E}$ for every $E\in \mathcal{Z}_\mathcal{N}(f)$, then $E\in \mathcal{Z}_\mathcal{N}[M]$.
Hence $E\in \mathcal{Z}_\mathcal{N}[M]$ for each $E\in \mathcal{Z}_\mathcal{N}(f)$; 
thus $\mathcal{Z}_\mathcal{N}(f)\subseteq \mathcal{Z}_\mathcal{N}[M]$, which implies that $f\in \mathcal{Z}_\mathcal{N}^{-1}\mathcal{Z}_\mathcal{N}[M]=M$ (as $M$ is a maximal ideal of $\mathcal{N}(X,\mathcal{A})$), i.e.\ $M\in N_f$ and so $\mathcal{Z}_\mathcal{N}[M]\in \mathcal{Z}_\mathcal{N}(N_f)$. Hence $\cap \{ \bar{E}: E\in \mathcal{Z}_\mathcal{N}(f)\}\subseteq \mathcal{Z}_\mathcal{N}[N_f]$. 
\end{proof}

We have previously observed that each ideal in $\mathcal{M}(X,\mathcal{A})$ is a $z^\circ$-ideal (indeed $\mathcal{M}(X,\mathcal{A})$ is a Von-Neumann regular ring). We shall now show that this property characterizes $\mathcal{M}(X,\mathcal{A})$ among intermediate rings of real valued measurable functions on $(X,\mathcal{A})$. 

\begin{theorem}\label{34}
An intermediate ring $\mathcal{N}(X,\mathcal{A})$ becomes identical to $\mathcal{M}(X,\mathcal{A})$ if and only if each ideal of $\mathcal{N}(X,\mathcal{A})$ is a $z^\circ$-ideal.
\end{theorem}

\begin{proof}
Suppose each ideal of $\mathcal{N}(X,\mathcal{A})$ be a $z^\circ$-ideal. 
We claim that for any $f\in \mathcal{M}(X,\mathcal{A})$, the function $\frac{1}{1+\lvert f\rvert}$ is a unit in $\mathcal{N}(X,\mathcal{A})$.
It would follow that $\lvert f\rvert \in \mathcal{N}(X,\mathcal{A})$ and consequently $f\in \mathcal{N}(X,\mathcal{A})$.
To prove the claim suppose toward a contradiction that there is an $f\in \mathcal{M}(X,\mathcal{A})$ such that $\frac{1}{1+\lvert f\rvert}$ is not a unit in $\mathcal{N}(X,\mathcal{A})$. Since $\frac{1}{1+\lvert f\rvert}\in \mathcal{N}(X,\mathcal{A})$, it follows that the principal ideal $(\frac{1}{1+\lvert f\rvert})$ in $\mathcal{N}(X,\mathcal{A})$ is a proper ideal and is hence a $z^\circ$-ideal by hypothesis. But $\frac{1}{1+\lvert f\rvert}$ is clearly not a divisor of zero in $\mathcal{N}(X,\mathcal{A})$.  Since each element of a $z^\circ$-ideal in a reduced ring is a divisor of zero a fact easily verifiable, this is a contradiction.
\end{proof}

\begin{corollary}\label{35}
An intermediate ring $\mathcal{N}(X,\mathcal{A})$ is Von-Neumann regular if and only if $\mathcal{N}(X,\mathcal{A})=\mathcal{M}(X,\mathcal{A})$.
\end{corollary} 

\begin{proof}
	If $\mathcal{N}(X,\mathcal{A})\subsetneq \mathcal{M}(X,\mathcal{A})$, then by Theorem \ref{34} there exists an ideal $I$ in $\mathcal{N}(X,\mathcal{A})$ which is not a $z^\circ$-ideal. Since in a Von-Neumann regular ring each (proper) ideal is a $z^\circ$-ideal, it follows that $\mathcal{N}(X,\mathcal{A})$ is not a regular ring.
\end{proof}	 
We have observed earlier (vide Theorem \ref{2}) that for $f,g\in \mathcal{M}(X,\mathcal{A})$, $Z(f)\supseteq Z(g)$ if and only if $f$ is a multiple of $g$.
The following result indicates that this fact also characterizes $\mathcal{M}(X,\mathcal{A})$ among the intermediate rings.

\begin{theorem}\label{36}
 Let $\mathcal{N}(X,\mathcal{A}) (\subsetneq \mathcal{M}(X,\mathcal{A}))$ be an intermediate ring of measurable functions on the measurable space $(X,\mathcal{A})$. Then there exist $g,h\in \mathcal{N}(X,\mathcal{A})$ such that $Z(g)\supseteq Z(h)$ but $g$ is not a multiple of $h$ in the ring $\mathcal{N}(X,\mathcal{A})$.
 \end{theorem} 

\begin{proof}
	We can choose $f\in \mathcal{M}(X,\mathcal{A})\setminus \mathcal{N}(X,\mathcal{A})$. Take $g=\frac{f}{1+\lvert f\rvert}$ and $h=\frac{1}{1+\lvert f\rvert}$. Then $g$ and $h$ are both bounded functions on $X$ and hence $g,h\in \mathcal{N}(X,\mathcal{A})$. We observe that $Z(g)\supseteq Z(h)=\emptyset$. But we claim that $g$ is not a multiple of $h$ in this ring $\mathcal{N}(X,\mathcal{A})$.
	To prove this claim, suppose there exists $k\in \mathcal{N}(X,\mathcal{A})$ with the relation $g=hk$. This means 
\begin{equation}\label{eq:frak}
\frac{f}{1+\lvert f\rvert}=\frac{k}{1+\lvert f\rvert}.
\end{equation} Since all the functions are real valued, on multiplying both sides of \eqref{eq:frak} by $1+\lvert f\rvert$, we get $ f=k$.
But this is a contradiction since $f\notin \mathcal{N}(X,\mathcal{A})$ while $k\in \mathcal{N}(X,\mathcal{A})$.
\end{proof}

\section{$P$-spaces and continuous functions}
\label{sec:openproblems}
The ring and lattice structures of $\mathcal{M}(X,\mathcal{A})$ share a number of properties possessed by the lattice ordered ring $C(Y)$ of all real valued continuous functions defined over a $P$-space $Y$. Here are some of the properties shared by both $M(X,\mathcal{A})$ and $C(Y)$: 
every prime ideal is maximal (Corollary \ref{3} and \cite[\S4J]{GJ});
each ideal is a $z^\circ$-ideal (a consequence of the previous property);
each ideal is closed when the $m$-topology is imposed on the ring (Theorem \ref{16} and \cite[\S7Q4]{GJ});
an ideal is maximal if and only if its residue class ring is totally ordered (Theorem \ref{19} and \cite[\S5P]{GJ};
the structure space is a compact Hausdorff zero-dimentional space (Theorem \ref{7} and \cite[\S7N]{GJ} in light of the fact that $\beta Y$ is basically disconnected and in particular zero-dimensional); as well as other properties.

The following proposition shows that when $Y$ is a $P$-space, then for some specific choice of the measurable space $(X,\mathcal{A})$, the ring $\mathcal{M}(X,\mathcal{A})$ is identical to $C(Y)$.

\begin{theorem}\label{thm:Pspaces}
If $X$ is a $P$-space, then the set $Z(X)$ of all zero-sets in $X$ is a $\sigma$-algebra on $X$, and if $\mathcal{A} = Z(X)$, then $\mathcal{M}(X,\mathcal{A})=C(X)$ and $\mathcal{M}^*(X,\mathcal{A})=C^*(X)$.
\end{theorem}

\begin{proof}
Since the zero-sets and cozero-sets of the $P$-space $X$ are one and the same (see \cite[\S J(3)]{GJ}) and $Z(X)$ is closed under countable intersection, it follows clearly that $Z(X)$ is a $\sigma$-algebra on $X$.

If $f\in \mathcal{M}(X,\mathcal{A})$, then for any open set $G$ in $\mathbb{R}$, $f^{-1}(G)$ is a member of $\mathcal{A}=Z(X)$, in particular $f^{-1}(G)$ is open in $X$.  Hence $f\in C(X)$. 
Conversely, let $f\in C(X)$, and $a\in \mathbb{R}$, then $(-\infty,a]$ is a closed set and hence a zero set in $\mathbb{R}$. Since the preimage of
 a zero set under a continuous map is a zero set, it follows that $f^{-1}(-\infty,a]$ is a zero set in $X$ and therefore a member of $\mathcal{A}$. Thus $f$ turns out to be a measurable function i.e. $f\in \mathcal{M}(X,\mathcal{A})$. 
Hence $\mathcal{M}(X,\mathcal{A}) = C(X)$ and $\mathcal{M}^*(X,\mathcal{A}) = C^*(X)$ immediately follows.
\end{proof}

Theorem \ref{thm:Pspaces} contributes to Question \ref{quest:realfixed} raised earlier, by clarifying that it is possible for there to be a free real maximal ideal of a ring of measurable functions, as the next example shows.
\begin{example}\label{ex:realmaxnotfixed}
There is a $P$-space $X$ that is not realcompact given in \cite[Exercise 9L]{GJ}.
Then $C(X)$ has a free real maximal ideal.
By Theorem \ref{thm:Pspaces}, $C(X) = \mathcal{M}(X,Z[X])$, and hence $\mathcal{M}(X,Z[X])$ has a free real maximal ideal.
\end{example} 

A question that naturally arises from Theorem \ref{thm:Pspaces} is as follows.
\begin{question}
Given a (possibly infinite) $\sigma$-algebra $\mathcal{A}$ on a set $X$, does there exist a $P$-space $Y$ such that the ring or equivalently the lattice structure of $\mathcal{M}(X,\mathcal{A})$ and $C(Y)$ are isomorphic?
\end{question}
One possible candidate for the space $Y$ is the set $X$, equipped with the weak topology induced by $\mathcal{M}(X,\mathcal{A})$; as the characteristic functions of measurable sets are in $\mathcal{M}(X,\mathcal{A})$, this topology is the smallest topology containing $\mathcal{A}$. But we show by way of a counterexample below that for such a choice of $Y$, even if $Y$ is a $P$-space, it may happen that $\mathcal{M}(X,\mathcal{A})$ is not isomorphic to $C(Y)$.  

\begin{example}
Let $X=[0,1]$ and $\mathcal{A}$ be the $\sigma$-algebra of all Lebesgue measurable sets on $[0,1]$. Let $\tau$ be the smallest topology on $X$ that contains $\mathcal{A}$ (equivalently, the weak topology on $X$ induced by $\mathcal{M}(X,\mathcal{A})$). 
Since every one-point set is Lebesgue measurable, all of the singleton sets are open.
Hence $(X,\tau)$ is the discrete topological space.
Then $C(X,\tau)$ consists of \emph{all} real-valued functions on $[0,1]$, and is hence distinct from $\mathcal{M}(X,\mathcal{A})$. 

To see that $C(X,\tau)$ and $\mathcal{M}(X,\mathcal{A})$ are not even isomorphic,
we look at their lattice structure. 
Since $Y$ is discrete, it is in particular extremally disconnected, meaning that every zero-set has an open closure. Consequently by Stone-Nakano's theorem (\cite[\S3N6]{GJ}), $C(Y)$ is a conditionally complete lattice in the sense that each nonempty subset of $C(Y)$, with an upper bound in $C(Y)$ has a supremum also lying in $C(Y)$. But we show that the lattice $\mathcal{M}(X,\mathcal{A})$ is not conditionally complete and hence $\mathcal{M}(X,\mathcal{A})$ and $C(Y)$ are not isomorphic as lattices, and consequently not isomorphic as rings. Indeed for each point $s$ lying on a fixed non Lebesgue measurable set $S$ in $[0,1]$, let $\chi_{\{s\}}$ be the characteristic function of $\{s\}$. Surely $\{\chi_{\{s\}}:s\in S\}$ is a subfamily of $\mathcal{M}(X,\mathcal{A})$, which is bounded above in this ring by the constant function 1. However $\sup \{ \chi_{\{s\}}:s\in S\}$ does not exist in $\mathcal{M}(X,\mathcal{A})$, an easy verification.
\end{example}
This example raises the following questions.
\begin{question}
Under what conditions on a $\sigma$-algebra $\mathcal{A}$ on $X$, is the ring $\mathcal{M}(X,\mathcal{A})$ isomorphic to the ring $C(X,\tau)$, where $\tau$ is the smallest topology on $X$ containing $\mathcal{A}$? 
\end{question}
\begin{question}
Under what conditions is the weak topology on $X$, induced by $\mathcal{M}(X,\mathcal{A})$ a $P$-space?
\end{question}

\section{Acknowledgement}
The authors would like to thank Professor Alan Dow for suggesting us the proof of Lemma \ref{14}.

\bibliographystyle{plain}

\end{document}